\documentclass{amsart}

\usepackage{amsmath}
\usepackage{amssymb}
\usepackage{graphicx}

\newtheorem{thm}{Theorem}
\newtheorem{cor}[thm]{Corollary}
\newtheorem{lem}[thm]{Lemma}
\newtheorem{propo}[thm]{Proposition}

\theoremstyle{definition}

\theoremstyle{remark}

\begin{document}

\title{Fourier series of Jacobi-Sobolev polynomials}

\author[\'O. Ciaurri and J. M\'{\i}nguez]{\'Oscar Ciaurri and Judit M\'{\i}nguez}

\address[\'O. Ciaurri]{Departamento de Matem\'aticas y Computaci\'on\\
	Universidad de La Rioja\\
	26006 Logro\~no, Spain}
\email{oscar.ciaurri@unirioja.es}

\address[J. M\'{\i}nguez]{Departamento de Matem\'aticas y Computaci\'on\\
	Universidad de La Rioja\\
	26006 Logro\~no, Spain}
\email{judit.minguez@unirioja.es}

\thanks{The authors were supported by grant MTM2015-65888-C04-4-P from Spanish Government.}

\begin{abstract}
Let $\{q_n^{(\alpha,\beta,m)}(x)\}_{n\ge 0}$ be the orthonormal polynomials respect to the Sobolev-type inner product
\begin{equation*}
\langle f,g\rangle_{\alpha,\beta,m}=\sum_{k=0}^m \int_{-1}^{1}f^{(k)}(x)g^{(k)}(x)\, dw_{\alpha+k,\beta+k}(x), \quad \alpha,\beta>-1, \quad m\ge 1,
\end{equation*}
where $dw_{a,b}(x)=(1-x)^{a}(1+x)^b\, dx$. We obtain necessary and sufficient conditions for the uniform boundedness of the partial sum operators related to this sequence of polynomials in the Sobolev space $W_{\alpha,\beta}^{p,m}$. As a consequence we deduce the convergence of such partial sums in the norm of $W_{\alpha,\beta}^{p,m}$.
\end{abstract}

\maketitle

\section{Introduction}
The study of orthogonal polynomials with respect to the Sobolev-type inner product
\begin{equation}
\label{pro-interno-general}
\langle f,g\rangle=\sum_{k=0}^m\int_{\mathbb{R}}f^{(k)}g^{(k)}\,d\mu_k,
\end{equation}
has attracted the interest of many researchers in the last years (see, for example, the survey \cite{Marce-Xu} and the references therein). In this paper we contribute to that study with the analysis of the Fourier series in terms of orthonormal polynomials associated with a particular Sobolev-type inner product. Specifically, for each $m\in \mathbb{N}\setminus\{0\}$, we consider the inner product
\begin{equation}
\label{prod-interno}
\langle f,g\rangle_{\alpha,\beta,m}=\sum_{k=0}^m \int_{-1}^{1}f^{(k)}(x)g^{(k)}(x)\, d\mu_{\alpha+k,\beta+k}(x), \qquad \alpha,\beta>-1,
\end{equation}
where $d\mu_{a,b}(x)=(1-x)^{a}(1+x)^b\, dx$. We exclude the case $m=0$ of our analysis because it corresponds with the classical inner product related to Jacobi polynomials.

By using the Rodrigues formula, the Jacobi polynomials $\{P^{(\alpha,\beta)}_n(x)\}_{n\ge 0}$ are defined as
\[
P_n^{(\alpha,\beta)}(x)=\frac{(-1)^n}{2^n \, n!}(1-x)^{-\alpha}(1+x)^{-\beta}\frac{d^n}{dx^n}\left\{(1-x)^{\alpha+n}(1+x)^{\beta+n}\right\}.
\]
They are orthogonal in the interval $[-1,1]$ with the measure $d\mu_{\alpha,\beta}$ and then the sequence $\{p_n^{(\alpha,\beta)}(x)\}_{n\ge 0}$, given by $p_n^{(\alpha,\beta)}(x)=w_n^{(\alpha,\beta)}P_n^{(\alpha,\beta)}(x)$ and
\[
w_n^{(\alpha,\beta)}=
\frac{1}{\|P_n^{(\alpha,\beta)}\|_{L^2([-1,1],dw_{\alpha,\beta})}}
=\sqrt{\frac{(2n+\alpha+\beta+1)\, n!\,\Gamma(n+\alpha+\beta+1)}
{2^{\alpha+\beta+1}\Gamma(n+\alpha+1)\,\Gamma(n+\beta+1)}},
\]
is orthonormal  and complete in $L^2([-1,1],dw_{\alpha,\beta})$. Moreover, the Jacobi polynomials are eigenfunctions of the second order differential operator
\[
L_{\alpha,\beta}f(x)=(1-x^2)f''(x)+((\beta+1)(1-x)-(\alpha+1)(1+x))f'(x).
\]
In fact,
\[
L_{\alpha,\beta}p_n^{(\alpha,\beta)}(x)=-\lambda_n^{(\alpha,\beta)}p_n^{(\alpha,\beta)}(x),
\]
with
\begin{equation}
\label{eq:eigenvalue}
\lambda_n^{(\alpha,\beta)}=n(n+\alpha+\beta+1).
\end{equation}

The identity (see \cite[p. 63, eq.~(4.21.7)]{Szego})
\[
\frac{d}{dx}P_n^{(\alpha,\beta)}(x)=\frac{n+\alpha+\beta+1}{2}P_{n-1}^{(\alpha+1,\beta+1)}(x),
\]
taking into account that $w_n^{(\alpha,\beta)}=2\sqrt{\frac{n}{n+\alpha+\beta+1}}w_{n-1}^{(\alpha+1,\beta+1)}$,
implies
\[
\frac{d}{dx}p_n^{(\alpha,\beta)}(x)=\sqrt{\lambda_{n}^{(\alpha,\beta)}}p_{n-1}^{(\alpha+1,\beta+1)}(x)
\]
and, more generally,
\begin{equation}
\label{eq:derivative}
\frac{d^k}{dx^k}p_n^{(\alpha,\beta)}(x)=\sqrt{r_{n,k}^{(\alpha,\beta)}}p_{n-k}^{(\alpha+k,\beta+k)}(x)
\end{equation}
where
\[
r_{n,k}^{(\alpha,\beta)}=\prod_{j=0}^{k-1}\lambda_{n-j}^{(\alpha+j,\beta+j)}, \qquad k\ge 1,
\]
and $r_{n,0}=1$. In this way, the polynomials
\[
q_{n}^{(\alpha,\beta,m)}(x)=\frac{p_n^{(\alpha,\beta)}(x)}{\sqrt{s^{(\alpha,\beta)}_{n,m}}},
\]
with
\[
s_{n,m}^{(\alpha,\beta)}=\sum_{k=0}^m r_{n,k}^{(\alpha,\beta)},
\]
are orthonormal respect the Sobolev-type inner product $\langle \cdot, \cdot \rangle_{\alpha,\beta,m}$; i. e., they satisfy
\[
\langle q_{n}^{(\alpha,\beta,m)},q_{j}^{(\alpha,\beta,m)}\rangle_{\alpha,\beta,m}=\delta_{n,j}.
\]

Given $1\le p<\infty$, we will write $L^p_{\alpha,\beta}$ to denote $L^p([-1,1],d\mu_{\alpha,\beta})$, the space of all measurable functions on $[-1,1]$ for wich
\[
\|f\|_{L^p_{\alpha,\beta}}:=\left(\int_{-1}^1|f(x)|^p\, d\mu_{\alpha,\beta}(x)\right)^{1/p}<\infty.
\]
For $p=\infty$, we consider the standard definition in terms of essential supremum.
We define the space $W^{p,m}_{\alpha,\beta}$, for $1\le p<\infty$, as the space of measurable functions $f$ defined on $[-1,1]$ such that there exist $f',f'',\dots,f^{(m)}$ almost everywhere and
\[
\|f\|_{W^{p,m}_{\alpha,\beta}}:=\left(\sum_{k=0}^m \|f^{(k)}\|_{L^p_{\alpha+k,\beta+k}}\right)^{1/p}<\infty.
\]

We denote by $S_n^{(\alpha,\beta,m)}f$ the $n$-th partial sum operator as
\[
S_n^{(\alpha,\beta,m)}f=\sum_{j=0}^n c_j^{(\alpha,\beta,m)}(f)q_j^{(\alpha,\beta,m)}(x),
\]
where
\[
c_j^{(\alpha,\beta,m)}(f)=\langle f,q_j^{(\alpha,\beta,m)}\rangle_{\alpha,\beta,m}
\]
are the Fourier-Jacobi-Sobolev coefficients. Our main result characterizes the uniform boundedness of the operators $S_n^{(\alpha,\beta,m)}$ in the spaces $W_{\alpha,\beta}^{p,m}$. In fact, we will prove the following theorem.

\begin{thm}
\label{thm:main}
Let $f\in W^{p,m}_{\alpha,\beta}$, with $\alpha\ge \beta>-1$, $m\in \mathbb{N}\setminus\{0\}$ and $1<p<\infty$. Then
\begin{equation}
\label{eq:acot-main}
\|S_n^{(\alpha,\beta,m)}f\|_{W^{p,m}_{\alpha,\beta}}\le C\|f\|_{W^{p,m}_{\alpha,\beta}}
\end{equation}
with a constant $C$ independent of $n$ and $f$, if and only if
\begin{equation}
\label{eq:alpha-cond}
	\frac{4(\alpha+m+1)}{2(\alpha+m)+3}<p<\frac{4(\alpha+m+1)}{2(\alpha+m)+1}.
\end{equation}
\end{thm}

The restriction $\alpha\ge\beta$ is imposed to simplify the proof of the result but we do not lose generality with it. In fact, when $\beta\ge \alpha$ the uniform boundedness \eqref{eq:acot-main} holds if and only if
\begin{equation*}
	\frac{4(\beta+m+1)}{2(\beta+m)+3}<p<\frac{4(\beta+m+1)}{2(\beta+m)+1},
\end{equation*}
and, in general, for $\alpha,\beta>-1$ \eqref{eq:acot-main} is verified if and only if
\[
\max\left\{\frac{4(\alpha+m+1)}{2(\alpha+m)+3}, \frac{4(\beta+m+1)}{2(\beta+m)+3}\right\}<p<
\min\left\{\frac{4(\alpha+m+1)}{2(\alpha+m)+1}, \frac{4(\beta+m+1)}{2(\beta+m)+1}\right\}.
\]

The analysis of the Fourier series of Jacobi polynomials has a long history. Pollard in \cite{PollardII} and \cite{PollardIII} studied the uniform boundedness of the partial sums for the Fourier series of Gegenbauer and Jacobi polynomials, respectively.  A general result including weights for Jacobi expansions can be seen in \cite{muckenhoupt1}. In \cite{GPRV1}, by applying the boundedness with weights of the Hilbert transform, the authors did a complete study of the boundedness of the partial sum operators related to generalized Jacobi weights. The same authors studied the  generalized Jacobi weights with mass points on the interval $[-1,1]$ (see \cite{GPRV2}).

In \cite{CM}, the authors gave a complete characterization of the uniform boundedness of the partial sum operators for the Fourier series related to orthonormal polynomials with respect to  \eqref{pro-interno-general} where $d\mu_0=d\mu_{\alpha}+M(\delta_1+\delta_{-1})$, with $d\mu_{\alpha}$ the probability measure corresponding to the Gegenbauer polynomials, $d\mu_1=N(\delta_1+\delta_{-1})$, and $d\mu_k=0$, $k\ge 2$. That was the first result of this type in the literature. In fact, as it is observed in \cite{Marce-Xu}, the main obstacle to analyze this kind of problems is the lack of a Christoffel–Darboux formula for Sobolev orthogonal polynomials. So it is necessary to look for alternative ways to deal with the problem.

In \cite{MQU}, the authors considered the Fourier series for polynomials associated to the Sobolev-type inner product
\begin{equation}
\label{prod-interno-gegen}
\langle f,g\rangle_S=\int_{-1}^{1}f(x)g(x)w_{\alpha}(x)\,dx+\int_{-1}^{1}f'(x)g'(x)w_{\alpha+1}(x)\,dx,
\end{equation}
where $w_{\alpha}(x)=(1-x^2)^{\alpha}$, $x\in [-1,1]$ and $\alpha>-1$. They analyzed the uniform boundedness of the partial sum operators for the orthogonal polynomials with respect \eqref{prod-interno-gegen} using the Pollard decomposition but, unfortunately, the given results are not completely satisfactory. Theorem \ref{thm:main} gives, as a particular case, necessary and sufficient conditions for this case.

Due to the denseness of the polynomials in the spaces $W_{\alpha,\beta}^{p,m}$ \cite{rodriguez}, applying the uniform boundedness theorem in the complete space $W_{\alpha,\beta}^{p,m}$, it is verified that \eqref{eq:acot-main} is equivalent to the convergence of the partial sums $S_n^{(\alpha,\beta,m)}$ in the spaces $W_{\alpha,\beta}^{p,m}$. Then, we have the following result.

\begin{cor} Let $f\in W_{\alpha,\beta}^{p,m}$ with $\alpha\ge \beta>-1$, $m\in \mathbb{N}\setminus\{0\}$, and $1<p<\infty$. Then
\begin{equation*}
	\lim_{n\to\infty}\|S_n^{(\alpha,\beta,m)}f-f\|_{W_{\alpha,\beta}^{p,m}}=0
\end{equation*}
if and only if 
\begin{equation*}
	\frac{4(\alpha+m+1)}{2(\alpha+m)+3}<p<\frac{4(\alpha+m+1)}{2(\alpha+m)+1}.
\end{equation*}
\end{cor}

The next section contains the proof of Theorem \ref{thm:main} and it is divided into two subsections, one for the sufficient conditions and the other for the necessary ones. The last section is devoted to the proof of a technical result involved in the proof of Theorem \ref{thm:main}.
\section{Proof of Theorem \ref{thm:main}}
\subsection{Sufficient conditions}

For $n$ big enough, it is clear that
\[
S_n^{(\alpha,\beta,m)}f(x)=\sum_{k=0}^{m}\mathcal{S}_n^{(\alpha,\beta,k)}f(x),
\]
where
\[
\mathcal{S}_n^{(\alpha,\beta,k)}f(x)=\sum_{j=k}^n \sqrt{\frac{r^{(\alpha,\beta)}_{j,k}}{s^{(\alpha,\beta)}_{j,m}}}b_j^{(\alpha,\beta,k)}(f^{(k)})q_j^{(\alpha,\beta,m)}(x)
\]
and
\[
b_j^{(\alpha,\beta,k)}(f^{(k)})=\int_{-1}^{1}f^{(k)}(y)p_{j-k}^{(\alpha+k,\beta+k)}(y)\, d\mu_{\alpha+k,\beta+k}(y).
\]
To obtain the uniform boundedness \eqref{eq:acot-main}, it is enough to prove that
\begin{equation}
\label{eq:partial-Lp}
\left\|\left(\mathcal{S}_n^{(\alpha,\beta,k)}f\right)^{(\ell)} \right\|_{L^p_{\alpha+\ell,\beta+\ell}}\le C \|f^{(k)}\|_{L^{p}_{\alpha+k,\beta+k}}, \qquad 0\le k, \ell \le m,
\end{equation}
under the conditions \eqref{eq:alpha-cond}. First, note that
\[
\left(\mathcal{S}_n^{(\alpha,\beta,k)}f\right)^{(\ell)}=\sum_{j=\max\{k,\ell\}}^n \frac{\sqrt{r^{(\alpha,\beta)}_{j,k}r^{(\alpha,\beta)}_{j,\ell}}}{s^{(\alpha,\beta)}_{j,m}}
b_j^{(\alpha,\beta,k)}(f^{(k)})p_{j-\ell}^{(\alpha+\ell,\beta+\ell)}.
\]
To obtain \eqref{eq:partial-Lp}, we distinguish three cases $2m-2\ge k+\ell$, $2m-1=k+\ell$, and $2m=k+\ell$.

\textbf{Case $2m-2\ge k+\ell$.} For $\alpha\ge \beta> -1$, it is well known (see \cite{LK}) the equivalence
\begin{equation}
\label{eq:Jacobi-Lp}
\|p_n^{(\alpha,\beta)}\|_{L^p_{\alpha,\beta}}\simeq \begin{cases}
1, & 1\le p <\frac{4(\alpha+1)}{2\alpha+1},\\[3pt]
(\log n)^{1/p}, & p=\frac{4(\alpha+1)}{2\alpha+1},\\[3pt]
n^{(2\alpha+1-4(\alpha+1)p)/2}, & p>\frac{4(\alpha+1)}{2\alpha+1}.
\end{cases}
\end{equation}
Then, using that
\begin{equation}
\label{eq:asym}
\frac{\sqrt{r^{(\alpha,\beta)}_{j,k}r^{(\alpha,\beta)}_{j,\ell}}}{s^{(\alpha,\beta)}_{j,m}}=\frac{1}{(j+1)^{2m-k-l}}
\left(A+\frac{B}{j+1}+O\left(\frac{1}{(j+1)^2}\right)\right),
\end{equation}
for some constants $A$ and $B$, we have
\[
\left\|\left(\mathcal{S}_n^{(\alpha,\beta,k)}f\right)^{(\ell)}\right\|_{L^p_{\alpha+\ell,\beta+\ell}}\le
\sum_{j=\max\{k,\ell\}}^n \frac{|b_j^{(\alpha,\beta,k)}(f^{(k)})|}{(j+1)^{2m-k-\ell}}
\|p_{j-\ell}^{(\alpha+\ell,\beta+\ell)}\|_{L^p_{\alpha+\ell,\beta+\ell}}.
\]
By H\"older inequality,
\[
|b_j^{(\alpha,\beta,k)}(f^{(k)})|\le \|p_{j-k}^{(\alpha+k,\beta+k)}\|_{L^{p'}_{\alpha+k,\beta+k}}\|f^{(k)}\|_{L^p_{\alpha+k,\beta+k}},
\]
where $p'$ is the conjugate value of $p$ and it satisfies $1/p+1/p'=1$, 
and
\begin{multline*}
\left\|\left(\mathcal{S}_n^{(\alpha,\beta,k)}f\right)^{(\ell)}\right\|_{L^p_{\alpha+\ell,\beta+\ell}}\\
\begin{aligned}
&\le
\|f^{(k)}\|_{L^p_{\alpha+k,\beta+k}}\sum_{j=\max\{k,\ell\}}^n \frac{\|p_{j-k}^{(\alpha+k,\beta+k)}\|_{L^{p'}_{\alpha+k,\beta+k}}
\|p_{j-\ell}^{(\alpha+\ell,\beta+\ell)}\|_{L^p_{\alpha+\ell,\beta+\ell}}}{(j+1)^{2m-k-\ell}}\\&\le
C\|f^{(k)}\|_{L^p_{\alpha+k,\beta+k}},
\end{aligned}
\end{multline*}
where in the last step we have used the inequality
\begin{multline*}
\|p_{j-k}^{(\alpha+k,\beta+k)}\|_{L^{p'}_{\alpha+k,\beta+k}}
\|p_{j-\ell}^{(\alpha+\ell,\beta+\ell)}\|_{L^p_{\alpha+\ell,\beta+\ell}}\\\le \|p_{j-k}^{(\alpha+m,\beta+m)}\|_{L^{p'}_{\alpha+m,\beta+m}}
\|p_{j-\ell}^{(\alpha+m,\beta+m)}\|_{L^p_{\alpha+m,\beta+m}},
\end{multline*}
\eqref{eq:Jacobi-Lp}, the restriction \eqref{eq:alpha-cond}, and the condition $2m-2\ge k+\ell$.

\textbf{Case $2m-1=k+\ell$.}
In this case we have to prove \eqref{eq:partial-Lp} for the pairs $(k,\ell)=(m,m-1)$ and $(k,\ell)=(m-1,m)$. We will focus in the last pair because the estimate for the other pair can be obtained by using similar arguments and duality. By \eqref{eq:asym}, we have
\[
\left(\mathcal{S}_n^{\alpha,\beta,m-1}f\right)^{(m)}(x)=A\mathcal{T}_1f^{(m-1)}(x)+\mathcal{T}_2f^{(m-1)}(x),
\]
for some constant $A$,
where
\[
\mathcal{T}_1f^{(m-1)}(x)=\sum_{j=1}^{n-m+1}\frac{b_{j+m-1}^{(\alpha,\beta,m-1)}(f^{(m-1)})}{j+m}p_{j-1}^{(\alpha+m,\beta+m)}(x)
\]
and
\[
|\mathcal{T}_2f^{(m-1)}(x)|\le C \sum_{j=m}^n\frac{|b_j^{(\alpha,\beta,m-1)}(f^{(m-1)})|}{(j+1)^2}|p_{j-m}^{(\alpha+m,\beta+m)}(x)|.
\]
The estimate
\[
\|\mathcal{T}_2f^{(m-1)}\|_{L^p_{\alpha+m,\beta+m}}\le C \|f^{(m-1)}\|_{L^{p}_{\alpha+m-1,\beta+m-1}},
\]
under the restrictions \eqref{eq:alpha-cond}, can be obtained by using H\"older inequality as in the previous case.
To prove the boundedness of $\mathcal{T}_1$ we write it as the composition of two operators. For $\alpha,\beta>0$, we define
\[
T_{\alpha,\beta}g(x)=\sum_{j=1}^{\infty}\frac{e_{j}^{(\alpha-1,\beta-1)}(g)}{j+m}p_{j-1}^{(\alpha,\beta)}(x)
\]
with
\[
e_j^{(\alpha,\beta)}(g)=\int_{-1}^1 g(y)p_j^{(\alpha,\beta)}(y)\, d\mu_{\alpha,\beta}.
\]
Moreover, for $\alpha,\beta>-1$, we consider the partial sum operator operator for the Jacobi expansions
\[
\mathbb{S}_n^{(\alpha,\beta)}h(x)=\sum_{j=0}^{n}e_j^{(\alpha,\beta)}(h)p_j^{(\alpha,\beta)}(x).
\]
It is known \cite{PollardIII} that, for $\alpha,\beta\ge -1/2$,
\[
\|\mathbb{S}_n^{(\alpha,\beta)}h\|_{L^p_{\alpha,\beta}}\le C \|h\|_{L^p_{\alpha,\beta}},
\]
with a constant $C$ independent of $n$ and $f$, if and only if
\begin{equation}
\label{eq:range-Sn}
\max\left\{\frac{4(\alpha+1)}{2\alpha+3},\frac{4(\beta+1)}{2\beta+3}\right\}<p<
\min\left\{\frac{4(\alpha+1)}{2\alpha+1},\frac{4(\beta+1)}{2\beta+1}\right\}.
\end{equation}
About the boundedness properties of the operator $T_{\alpha,\beta}$ we have the following result.
\begin{propo}
\label{propo-T}
  For $\alpha,\beta>0$ and $1<p<\infty$, it is verified that
  \begin{equation}
  \label{eq:acot-T}
  \|T_{\alpha,\beta}g\|_{L^p_{\alpha,\beta}}\le C\|g\|_{L^p_{\alpha-1,\beta-1}},
  \end{equation}
  for each $g\in L^p_{\alpha-1,\beta-1}$.
\end{propo}
The proof of this proposition is highly technical and it is postponed to the last section.

Now, it is easy to check that
\[
\mathcal{T}_1f^{(m-1)}(x)=T_{\alpha+m,\beta+m}(\mathbb{S}_{n-m+1}^{\alpha+m-1,\beta+m-1}f^{(m-1)})(x).
\]
Then, by Proposition \ref{propo-T} and \eqref{eq:range-Sn}, it is clear that
\[
\|\mathcal{T}_1f^{(m-1)}\|_{L^p_{\alpha+m,\beta+m}}\le C \|f^{(m-1)}\|_{L^p_{\alpha+m-1,\beta+m-1}}
\]
when the conditions \eqref{eq:alpha-cond} hold.

\textbf{Case $2m=k+\ell$.} In this last case $k=\ell=m$ and, by \eqref{eq:asym}, we have the decomposition
\[
(\mathcal{S}^{(\alpha,\beta,m)}_nf)^{(m)}(x)=A\mathbb{S}_{n-m}^{(\alpha+m,\beta+m)}f^{(m)}(x)+B\mathcal{P}_1f^{(m)}(x)+\mathcal{P}_2f^{(m)}(x),
\]
where
\[
\mathcal{P}_1f^{(m)}(x)=\sum_{j=0}^{n-m}\frac{e_j^{(\alpha+m,\beta+m)}(f^{(m)})}{j+m+1}p_j^{(\alpha+m,\beta+m)}(x)
\]
and
\[
|\mathcal{P}_2f^{(m)}(x)|\le C \sum_{j=m}^n\frac{|b_j^{(\alpha,\beta,m)}(f^{(m)})|}{(j+1)^2}|p_{j-m}^{(\alpha+m,\beta+m)}(x)|
\]
The estimate
\[
\|\mathcal{P}_2f^{(m)}\|_{L^p_{\alpha+m,\beta+m}}\le C \|f^{(m)}\|_{L^p_{\alpha+m,\beta+m}}
\]
when the conditions \eqref{eq:alpha-cond} hold, it is obtained by applying H\"older as in the two previous cases. When we consider  \eqref{eq:range-Sn} with $\alpha+m$ and $\beta+m$ instead of $\alpha$ and $\beta$, with $\alpha\ge \beta$, we obtain \eqref{eq:alpha-cond}, so we deduce that
\[
\|\mathbb{S}_{n-m}^{(\alpha+m,\beta+m)}f^{(m)}\|_{L^p_{\alpha+m,\beta+m}}\le C \|f^{(m)}\|_{L^p_{\alpha+m,\beta+m}}.
\]

Finally to analyze the operator $\mathcal{P}_1$ we need an auxiliary operator and its boundedness properties. We define
\[
\mathcal{R}^{(\alpha,\beta)} f(x)=\sum_{j=0}^\infty \frac{e_j^{(\alpha,\beta)}(f)}{j+m+1}p_{j}^{(\alpha,\beta)}(x).
\]
\begin{lem}\label{lem:muck}
Let $\alpha\ge\beta>-1$, $j\in \mathbb{N}\setminus \{0\}$, and
\[
\frac{4(\alpha+j+1)}{2(\alpha+j)+3}<p<\frac{4(\alpha+j+1)}{2(\alpha+j)+1}.
\]
Then,
	\begin{equation*}\label{muk}
	\|\mathcal{R}^{(\alpha+j,\beta+j)}f\|_{L^p_{\alpha+j,\beta+j}}\le C \|f\|_{L^p_{\alpha+j,\beta+j}}.
	\end{equation*}
\end{lem}
This lemma is a particular case of \cite[Theorem 1.10]{muckenhoupt} because the multiplier $1/(j+m+1)$ belongs to the class $M(1,1)$ there defined.

Now, it is clear that
\[
\mathcal{P}_1f^{(m)}(x)=\mathcal{R}^{(\alpha+m,\beta+m)}(\mathbb{S}_{n-m}^{(\alpha+m,\beta+m)}f^{(m)})(x)
\]
and the estimate
\[
\|\mathcal{P}_1f^{(m)}\|_{L^p_{\alpha+m,\beta+m}}\le C \|f^{(m)}\|_{L^p_{\alpha+m,\beta+m}}
\]
is an immediate consequence of the previous lemma and the boundedness of the partial sum operator for the Jacobi expansions.

\subsection{Necessary conditions}
If \eqref{eq:acot-main} holds, it is clear that
\begin{equation}
\label{eq:dual}
|c_n^{(\alpha,\beta,m)}(f)|\|q_{n}^{(\alpha,\beta,m)}\|_{W^{p,m}_{\alpha,\beta}}
=\|S_n^{(\alpha,\beta,m)}f-S_{n-1}^{(\alpha,\beta,m)}f\|_{W^{p,m}_{\alpha,\beta}}\le C\|f\|_{W_{\alpha,\beta}^{p,m}}
\end{equation}
By \cite[Theorem 4.3]{MQR}, each functional in $T\in (W^{p,m}_{\alpha,\beta})'$, with $1\le p<\infty$, can be written as
\[
T(f)=\sum_{k=0}^m \int_{-1}^{1}f^{(k)}(x)v_k(x)\, d\mu_{\alpha+k,\beta+k},
\]
where $v=(v_0,\dots,v_m)$ belongs to the space $\prod_{k=0}^{m}L^q_{\alpha+k,\beta+k}$ equipped with the norm
\[
\|v\|^p_{\prod_{k=0}^{m}L^q_{\alpha+k,\beta+k}}=\sum_{k=0}^{m}\|v_k\|^q_{L^q_{\alpha+k,\beta+k}}.
\]
Moreover, $\|T\|=\|v\|_{\prod_{k=0}^{m}L^q_{\alpha+k,\beta+k}}$ and the function $v$ is unique for $1<p<\infty$. From this fact, it is clear that the norm as operator of $c_n^{(\alpha,\beta,m)}(f)$ is given by $\|q_{n}^{(\alpha,\beta,m)}\|_{W^{q,m}_{\alpha,\beta}}$ and, by \eqref{eq:dual}, the inequality
\begin{equation}
\label{eq:dual2}
\|q_{n}^{(\alpha,\beta,m)}\|_{W^{p,m}_{\alpha,\beta}}\|q_{n}^{(\alpha,\beta,m)}\|_{W^{q,m}_{\alpha,\beta}}\le C
\end{equation}
holds with a constant independent of $n$ when \eqref{eq:acot-main} is verified.

Now, taking into account that
\[
\|q_{n}^{(\alpha,\beta,m)}\|_{W^{p,m}_{\alpha,\beta}}^p=\sum_{k=0}^{m}\sqrt{\frac{r_{n,k}^{(\alpha,\beta)}}{s^{(\alpha,\beta)}_{n,m}}}
\|p_{n-k}^{(\alpha+k,\beta+k)}\|_{L^p_{\alpha+k,\beta+k}}
\]
and the inequality $\|p_{n-k}^{(\alpha+k,\beta+k)}\|_{L^p_{\alpha+k,\beta+k}}\le C \|p_{n-k}^{(\alpha+m,\beta+m)}\|_{L^p_{\alpha+m,\beta+m}}$, for $0\le k \le m$, we can deduce that
\[
\|q_{n}^{(\alpha,\beta,m)}\|_{W^{p,m}_{\alpha,\beta}}\simeq \|p_{n-k}^{(\alpha+m,\beta+m)}\|_{L^p_{\alpha+m,\beta+m}}.
\]
Then, \eqref{eq:dual2} and \eqref{eq:Jacobi-Lp} imply \eqref{eq:alpha-cond}.
\section{Proof of the Proposition \ref{propo-T}}
It is easy to check that
\[
T_{\alpha,\beta}g(x)=\int_{-1}^{1}g(y)L(x,y)\,d\mu_{\alpha-1,\beta-1}(y),
\]
with
\[
L(x,y)=\sum_{j=1}^\infty\frac{p_{j-1}^{(\alpha,\beta)}(x)p_j^{(\alpha-1,\beta-1)}(y)}{j+m}
\]
Then, with the change of variable $x=\cos\theta$ and $y=\cos \omega$, the inequality \eqref{eq:acot-T} is equivalent to
\begin{equation}
\label{eq:acot-TT}
\int_{0}^{\pi}|\overline{T}_{\alpha,\beta}G(\theta)|^pW_{\alpha,\beta}(\theta)\, d\theta\\\le C
\int_{0}^{\pi}|G(\theta)|^pW_{\alpha-1,\beta-1}(\theta)\, d\theta,
\end{equation}
where $W_{\alpha,\beta}(\theta)=(\sin \theta/2)^{(\alpha+1/2)(2-p)}(\cos \theta/2)^{(\beta+1/2)(2-p)}$,
\[
\overline{T}_{\alpha,\beta}G(\theta)=\int_{-1}^{1}G(\omega)\mathcal{L}(\theta,\omega)\,d\omega,
\]
with
\[
\mathcal{L}(\theta,\omega)=\sum_{j=1}^\infty\frac{\phi_{j-1}^{(\alpha,\beta)}(\theta)\phi_{j}^{(\alpha-1,\beta-1)}(\omega)}{j+m}
\]
and
\[
\phi_k^{(a,b)}(\theta)=2^{(\alpha+\beta+1)/2}(\sin \theta/2)^{\alpha+1/2}(\cos\theta/2)^{\beta+1/2}p_k^{(a,b)}(\cos\theta).
\]

Now, for any integer $d$ and $0<r<1$, we consider the auxiliary kernel
\[
\mathcal{L}^{(\alpha,\beta),(\alpha-1,\beta-1)}_{r,d,m}(\theta,\omega)=\sum_{j=\max\{0,-d\}}^\infty r^j \frac{\phi_{j+d}^{(\alpha,\beta)}(\theta)\phi_{j}^{(\alpha-1,\beta-1)}(\omega)}{j+m+1}.
\]
\begin{lem}
For $\alpha,\beta>0$, $0<r<1$, and $0<\theta,\omega<\pi$, it is verified that
\begin{equation}
\label{eq:bound-L}
|\mathcal{L}_{r,-1,m}^{(\alpha,\beta),(\alpha-1,\beta-1)}(\theta,\omega)|\le C \begin{cases}
\dfrac{\omega^{\alpha-1/2}(\pi-\theta)^{\beta+1/2}}{\theta^{\alpha-1/2}(\pi-\omega)^{\beta+1/2}}, & 0<\omega\le M(\theta),\\[6pt]
\log \bigg(\dfrac{2\theta}{|\theta-\omega|}\bigg), & M(\theta)<\omega<m(\theta),\\[6pt]
\dfrac{\theta^{\alpha+1/2}(\pi-\omega)^{\beta-1/2}}{\omega^{\alpha+1/2}(\pi-\theta)^{\beta-1/2}}, & m(\theta)\le \omega <\pi,
\end{cases}
\end{equation}
with
\[
M(\theta)=\max\left\{\frac{\theta}{2},\frac{3\theta-\pi}{2}\right\}\qquad \text{ and }\qquad m(\theta)=\min\left\{\frac{3\theta}{2},\frac{\theta+\pi}{2}\right\}
\]
\end{lem}
\begin{proof}
To prove the estimate in the first line of \eqref{eq:bound-L} note that when $0<\theta,\omega<3\pi/4$, the right-hand side is equivalent to $\omega^{\alpha-1/2}\theta^{1/2-\alpha}$ and this can be deduced from \cite[Theorem 7.1]{muckenhoupt} (in fact, we have to consider in that theorem $d=-1$, $s=1$, and $g(j)=1/(j+m)$). When $\pi/4<\theta,\omega<\pi$, the required estimate is comparable with $(\pi-\theta)^{\beta+1/2}(\pi-\omega)^{-\beta-1/2}$, and the required bound is obtained by using the identities (remember that $P_n^{(a,b)}(-x)=(-1)^n P_n^{(b,a)}(x)$)
\[
\mathcal{L}_{r,-1,m}^{(\alpha,\beta),(\alpha-1,\beta-1)}(\theta,\omega)=-\mathcal{L}_{r,-1,m}^{(\beta,\alpha),(\beta-1,\alpha-1)}(\pi-\theta,\pi-\omega)
=r\mathcal{L}_{r,1,m+1}^{(\alpha-1,\beta-1),(\alpha,\beta)}(\omega,\theta)
\]
and again \cite[Theorem 7.1]{muckenhoupt} (in this case with $d=1$, $s=1$, and $g(j)=1/(j+m+1)$). Finally, for $3\pi/4\le \theta<\pi$ and $0<\omega\le \pi/4$ the bound is equivalent to $\omega^{\alpha-1/2}(\pi-\theta)^{\beta+1/2}$ and this one is contained in \cite[Theorem 5.1]{muckenhoupt}.

The estimate in the third line of \eqref{eq:bound-L} is obtained in similar manner because is the dual of the bound in the first line.

To obtain the bound in the second line of \eqref{eq:bound-L} we proceed as in the analysis of the kernel $L_r^{1,-1}$ in \cite[Proposition 3.2, pp. 365--366]{CNS}.
\end{proof}

Following the ideas in the proof of \cite[Proposition 3.3]{CNS}, it is possible to prove that $\lim_{r\to 1^{-}}\mathcal{L}_{r,-1,m}^{(\alpha,\beta),(\alpha-1,\beta-1)}(\theta,\omega)$ exists and
\[
\mathcal{L}(\theta,\omega)=\lim_{r\to 1^{-}}\mathcal{L}_{r,-1,m}^{(\alpha,\beta),(\alpha-1,\beta-1)}(\theta,\omega).
\]
Moreover, $|\mathcal{L}(\theta,\omega)|$ is bounded by the right-hand side of \eqref{eq:bound-L}.
Then the operator $\overline{T}_{\alpha,\beta}G(\theta)$ can be controlled by the sum of
\[
R_1G(\theta)=\frac{(\pi-\theta)^{\beta+1/2}}{\theta^{\alpha-1/2}}\int_{0}^{\theta}\frac{\omega^{\alpha-1/2}}{(\pi-\omega)^{\beta+1/2}}|G(\omega)|\, d\omega,
\]
\[
R_2G(\theta)=\int_{M(\theta)}^{m(\theta)}\log \bigg(\dfrac{2\theta}{|\theta-\omega|}\bigg)|G(\omega)|\, d\omega,
\]
and
\[
R_3G(\theta)=\frac{\theta^{\alpha+1/2}}{(\pi-\theta)^{\beta-1/2}}\int_{\theta}^{\pi} \frac{(\pi-\omega)^{\beta-1/2}}{\omega^{\alpha+1/2}}|G(\omega)|\, d\omega.
\]

It is known, it is a consequence of \cite[Theorem A]{andersen}, that the inequality
\[
\int_{0}^{\pi}\left|U(\theta)\int_{0}^{\theta}h(\omega)\,d\omega\, \right|^p\, d\theta\le C \int_{0}^{\pi}\left|V(\theta)h(\theta)\right|^p\, d\theta
\]
holds if and only if
\begin{equation}
\label{eq:con-Hardy}
\sup_{0<r<\pi}\left(\int_{r}^{\pi}U^p(\theta)\, d\theta\right)^{1/p}\left(\int_{0}^{r}V^{-p'}(\theta)\, d\omega\right)^{1/p'}<\infty.
\end{equation}
Moreover, from \cite[Theorem B]{andersen}, we have
\[
\int_{0}^{\pi}\left|U(\theta)\int_{\theta}^{\pi}h(\omega)\,d\omega\, \right|^p\, d\theta\le C \int_{0}^{\pi}\left|V(\theta)h(\theta)\right|^p\, d\theta
\]
 if and only if
\begin{equation}
\label{eq:con-Hardy-ad}
\sup_{0<r<\pi}\left(\int_{0}^{r}U^p(\theta)\, d\theta\right)^{1/p}\left(\int_{r}^{\pi}V^{-p'}(\theta)\, d\omega\right)^{1/p'}<\infty.
\end{equation}

In this way the boundedness
\[
\int_{0}^{\pi}|R_1G(\theta)|^p W_{\alpha,\beta}(\theta)\le C \int_{0}^{\pi}|G(\theta)|^p W_{\alpha-1,\beta-1}(\theta)\, d\theta
\]
will follow checking the condition \eqref{eq:con-Hardy} for the weights
\[
U(\theta)=W_{\alpha,\beta}^{1/p}(\theta)\frac{(\pi-\theta)^{\beta+1/2}}{\theta^{\alpha-1/2}}
\qquad \text{ and }\qquad
V(\theta)=W_{\alpha-1,\beta-1}^{1/p}(\theta)\frac{(\pi-\theta)^{\beta+1/2}}{\theta^{\alpha-1/2}}.
\]
For these weights the supremum in \eqref{eq:con-Hardy} is equivalent to
\[
\sup_{0<r<\pi}\left(\int_{r}^{\pi}\theta^{2\alpha(1-p)+1}(\pi-\theta)^{2\beta+1}\, d\theta\right)^{1/p}\left(\int_{0}^{r}\theta^{2\alpha-1}(\pi-\theta)^{2\beta(1-p')-1}\, d\theta\right)^{1/p'}
\]
and this quantity is finite for $1<p<\infty$ and $\alpha,\beta>0$. To obtain the inequality
\[
\int_{0}^{\pi}|R_3G(\theta)|^p W_{\alpha,\beta}(\theta)\le C \int_{0}^{\pi}|G(\theta)|^p W_{\alpha-1,\beta-1}(\theta)\, d\theta
\]
we proceed in the same way but checking the condition \eqref{eq:con-Hardy-ad} for the appropriate weights.

To complete the proof of the inequality \eqref{eq:acot-TT} we have to check that
\[
\int_{0}^{\pi}|R_2G(\theta)|^p W_{\alpha,\beta}(\theta)\le C \int_{0}^{\pi}|G(\theta)|^p W_{\alpha-1,\beta-1}(\theta)\, d\theta.
\]
With some elementary manipulations, the previous inequality follows from
\[
\int_{0}^{\pi/2}\left|\int_{\theta/2}^{3\theta/2}\log \left(\frac{2\theta}{|\theta-\omega|}\right)h_1(\omega)\,d\omega\right|^p\theta^{(2-p)}\, d\theta\le \int_{0}^{\pi/2}|h_1(\theta)|^p\,d\theta
\]
and
\[
\int_{\pi/2}^{\pi}\left|\int_{(3\theta-\pi)/2}^{(\theta+\pi)/2}\log \left(\frac{2\theta}{|\theta-\omega|}\right)h_2(\omega)\,d\omega\right|^p(\pi-\theta)^{(2-p)}\, d\theta\le \int_{\pi/2}^{\pi}|h_2(\theta)|^p\,d\theta
\]
and both of them can be deduced applying H\"older inequality. Now, the proof of the proposition is finished.

\end{document}